\documentclass[a4paper,12pt]{article}

\usepackage[french, ngerman, italian, english]{babel}
\usepackage{amsmath,amsthm, amssymb}
\usepackage{setspace}
\usepackage{anysize}
\usepackage{url}
\usepackage{graphicx}
\usepackage[colorlinks=true]{hyperref}
\usepackage{mathptmx}
\usepackage{paralist}
\usepackage{verbatim}
\usepackage[utf8]{inputenc}

\theoremstyle{plain}
\newtheorem{thm}{\bf Theorem}[section]

\newtheorem{prop}[thm]{\bf Proposition}
\newtheorem{lemma}[thm]{\bf Lemma}
\newtheorem{corollary}[thm]{\bf Corollary}
\newtheorem{conjecture}[thm]{\bf Conjecture}

\theoremstyle{definition}
\newtheorem{definition}[thm]{\bf Definition}
\theoremstyle{remark}
\newtheorem{remark}[thm]{\bf Remark}
\newtheorem{example}[thm]{\bf Example}
\theoremstyle{example}
\numberwithin{equation}{section}

\def \Tor{{\operatorname{Tor}}}

\def \LT{{\operatorname{LT}_{\prec}}}
\def \Gens{{\operatorname{Gens}}}

\def \col{{\operatorname{col}}}
\def \link{{\operatorname{link}}}
\def \lk{{\operatorname{link}}}
\def \lk{{\operatorname{link}}}

\def \G{{\Gamma}}
\def \NN{\mathbb N}
\def\D{\Delta}

\begin{document}
\title{ On the  $h$-vectors of Cohen-Macaulay Flag Complexes}
\author{Alexandru Constantinescu \\
\footnotesize Dipartimento di Matematica\\
\footnotesize Univ. degli Studi di Genova, Italy\\
\footnotesize \url{constant@dima.unige.it}
\and \setcounter{footnote}{3} Matteo Varbaro \\
\footnotesize Dipartimento di Matematica\\
\footnotesize Univ. degli Studi di Genova, Italy\\
\footnotesize \url{varbaro@dima.unige.it}}
\date{{\small \today}} 

\maketitle
\abstract{ 
Starting from an unpublished conjecture of Kalai and from a conjecture of Eisenbud, Green and Harris, we study several problems relating $h$-vectors of Cohen-Macaulay, flag simplicial complexes and face vectors of simplicial complexes. 
}
\section{Introduction}

The $f$-vectors of simplicial complexes and the $h$-vectors of standard graded $K$-algebras are fascinating subjects in combinatorics and  commutative algebra. These topics have been the object of study for many researchers in the past decades (for instance see \cite{BFS,FFK,EGH,Fr}). The $f$-vectors of simplicial complexes have been completely characterized by Kruskal and Katona, and the $h$-vectors of Cohen-Macaulay standard graded $K$-algebras have been characterized by Macaulay. However, many questions regarding both $f$- and $h$-vectors  remain open, when extra properties are assumed for the simplicial complex, respectively for the standard graded algebra.

 
An unpublished conjecture of Kalai stated that for any flag simplicial complex there exists a balanced simplicial complex with the same $f$-vector. 
This fact has been recently  proven by Frohmader in \cite{Fr}. 
This conjecture of Kalai has also a second part which is still open, namely:
\begin{conjecture}[Kalai]\label{kalai} The following inclusion holds true:
\begin{equation}
\!\! \left \{ \! \!
	\begin{array}{c}
	\hbox{$f$-vectors of Cohen-Macaulay,} \\
	\hbox{ flag simplicial complexes}
	\end{array}
\right \}
\subseteq
\left \{ 
	\begin{array}{c}
	\hbox{$f$-vectors of Cohen-Macaulay,}  \\
 	\hbox{ balanced simplicial complexes}
	\end{array}
\!\!\!\right \}. \nonumber
\end{equation}
\end{conjecture}
\noindent Since the $h$-vectors of Cohen-Macaulay, balanced simplicial complexes are  $f$-vectors of simplicial complexes, the following would be a consequence of Kalai's Conjecture \ref{kalai}:
\begin{conjecture}\label{weakkalai}
The following inclusion holds true:
\begin{equation}
\!\! \left \{ \! \!
	\begin{array}{c}
	\hbox{$h$-vectors of Cohen-Macaulay,} \\
	\hbox{ flag simplicial complexes}
	\end{array}
\right \}
\subseteq
\left \{ 
	\begin{array}{c}
	\hbox{$f$-vectors of}  \\
 	\hbox{simplicial complexes}
	\end{array}
\!\!\!\right \}. \nonumber
\end{equation}
\end{conjecture}
\noindent Actually, the above inclusion is a particular case of a more general conjecture by Eisenbud, Green and Harris  (see \cite{EGH} or  the lecture notes by Valla \cite{Va}), which can be stated as:
\begin{conjecture}[Eisenbud, Green and Harris]\label{EGHconj} The following inclusion holds true:
\begin{equation}
\!\! \left \{ \! \!
	\begin{array}{c}
	\hbox{$h$-vectors of} \\
	\hbox{quadratic Artinian $K$-algebras}
	\end{array}
\right \}
\subseteq
\left \{ 
	\begin{array}{c}
	\hbox{$f$-vectors of}  \\
 	\hbox{simplicial complexes}
	\end{array}
\!\!\!\right \}. \nonumber
\end{equation}
\end{conjecture}

After introducing most of the terminology that we will need, in Section 2 we present a few  results and remarks that we will use throughout this paper. In particular, in Theorem \ref{structure}, we will extend results of Crupi, Rinaldo and Terai from \cite{CRT} and of the two authors from \cite{CV}.

In the third section  we will prove Conjecture \ref{weakkalai} for vertex decomposable, flag simplicial complexes (Theorem \ref{generalvd}). This section also includes an example of a $h$-vector of a quadratic Artinian algebra, which is the $f$-vector of a balanced complex, but not the $h$-vector of a Cohen-Macaulay,  flag  simplicial complex (Example \ref{balanced}). The section ends with a few comments on some technical aspects in  the proof  of Theorem \ref{generalvd}.

In Section 4 we will first notice that the $f$-vector of a flag simplicial complex is always the $h$-vector of a vertex decomposable, balanced, flag simplicial complex (Proposition \ref{fflaghbvd}). This result led us to the statement:
\begin{conjecture}\label{conjalexvabba} The following equality holds true:
\begin{equation}\label{eqalexvabba}
\!\! \left \{ \! \!
	\begin{array}{c}
	\hbox{$h$-vectors of vertex decomposable,} \\
	\hbox{balanced, flag simplicial complexes}
	\end{array}
\right \}
=
\left \{ 
	\begin{array}{c}
	\hbox{$f$-vectors of}  \\
 	\hbox{flag simplicial complexes}
	\end{array}
\!\!\!\right \}. 
\end{equation}\end{conjecture}
\noindent We were not able to find a proof for the above equality. However, relaxing the requests on the right hand side or strengthening the ones on the left we will be able to prove the hard inclusion of Conjecture \ref{conjalexvabba}. 
First, in Definition \ref{quasi-flag} we introduce a new class of simplicial complexes -- the quasi-flag simplicial complexes. It turns out that flag complexes are quasi-flag and in general the converse is not true. However, we are not aware of any quasi-flag simplicial complex whose $f$-vector is not the one of a flag simplicial complex.
We will then  prove the following inclusion (Theorem \ref{balanced-quasi-flag}):
\begin{equation}
\!\! \left \{ \! \!
	\begin{array}{c}
         \hbox{$h$-vectors of vertex decomposable,} \\
	\hbox{balanced, flag simplicial complexes} 
	\end{array}
\right \}
\subseteq
\left \{ 
	\begin{array}{c}
	\hbox{$f$-vectors of}  \\
 	\hbox{quasi-flag simplicial complexes}
	\end{array}
\!\!\!\right \} .\nonumber
\end{equation}
 
In the fifth section we are going to discuss a natural extension of Conjecture \ref{conjalexvabba}:
\begin{conjecture}\label{conjalexvabba1} The following equality holds true:
\begin{equation}
\!\! \left \{ \! \!
	\begin{array}{c}
	\hbox{$h$-vectors of Cohen-Macaulay,} \\
	\hbox{ flag simplicial complexes}
	\end{array}
\right \}
=
\left \{ 
	\begin{array}{c}
	\hbox{$f$-vectors of}  \\
 	\hbox{flag simplicial complexes}
	\end{array}
\!\!\!\right \}. \nonumber
\end{equation}
\end{conjecture}
\noindent In Proposition \ref{short} we will see that the above conjecture is true when the $h$-vector is of the form $(1,n,m)$.  We will then prove the  following result (Theorem \ref{dimcodim}):
\begin{equation}
\!\! \left \{ \! \!
	\begin{array}{c}
         \hbox{$h$-vectors of $(d-1)$-dimensional} \\
	\hbox{Cohen-Macaulay, flag simplicial complexes}  \\
 	\hbox{on $[2d]$,  without cone points}
	\end{array}
\right \}
=
\left \{ 
	\begin{array}{c}
	\hbox{$f$-vectors of flag}  \\
 	\hbox{ simplicial complexes on $[d]$}
	\end{array}
\!\!\!\right \} .\nonumber
\end{equation}
In a certain sense the above result is a first step towards proving Conjecture \ref{conjalexvabba1}. This is because  when $\D$ is a Cohen-Macaulay, $(d-1)$-dimensional, flag simplicial complex on $[n]$, without cone points, we have $n \ge 2d$. 

In the last section we will come back to Conjecture \ref{conjalexvabba}. We introduce two properties of simplicial complexes and show that for each of them, if added on the left hand side of \eqref{eqalexvabba}, the conjecture holds. We also include examples of simplicial complexes with or without these properties.

Many results in this paper have been suggested and double-checked by extensive computer algebra experiments performed  with CoCoA \cite{cocoa}.

The authors wish to thank Isabella Novik and Volkmar Welker for their useful suggestions and comments. We  also thank Aldo Conca for his support and helpful remarks. 

\section{Preliminaries} 
Let us start by introducing some terminology and notation that we will use throughout the paper. 
For general aspects on the topics presented below we refer the reader to the books 
of Stanley \cite{St} , of Bruns and Herzog \cite{BH} and of Lov\'asz and Plummer \cite{LP}.

For a positive integer $n$ denote by $[n]$ the set $\{1,\ldots,n\}$. A simplicial complex $\Delta$ on $[n]$ is a collection of subsets of $[n]$ such that $ F \in \Delta$ and $ F ' \subset  F$ imply $ F ' \in \Delta$. We will also require that for every $i \in [n]$ we have $\{i\} \in \Delta$. Each element $F\in\D$ is called a \emph{face} of $\D$. A maximal face of $\D$ with respect to inclusion is called a $\emph{facet}$ and  we will denote by $\mathcal{F}(\Delta)$ the set of facets of $\Delta$.  We call a vertex $v$ a \emph{cone point} of $\D$ if $v \in F$ for any $F \in \mathcal{F}(\D)$. A simplicial complex is called \emph{pure} if all facets have the same cardinality.
The dimension of a face $F$ is $|F| -1$ and the \emph{dimension of $\D$} is $\max\{\dim F ~:~ F \in \D\}$.  

Let $f_i = f_i(\D)$ denote the number of faces of $\D$ of dimension $i$, in particular $f_{-1}=1$ and $f_0= n$. The sequence $f(\D)=(f_{-1},f_{0},\ldots,f_{d-1})$, where $d-1$ is the dimension of $\D$, is called the \emph{$f$-vector} of $\D$. 

Denote by $S=K[x_1,\ldots,x_n]$ the polynomial ring in $n$ variables over a field $K$ and let $\D$ be a simplicial complex on $[n]$. For each subset $F\subset[n]$ we set 
\[\texttt{x}_{F} = \prod_{i\in F} x_i.\] 
The \emph{Stanley-Reisner ideal} of $\D$ is the ideal $I_{\D}$ of $S$  generated by the square-free monomials $\texttt{x}_{F},$ with $F \notin\D$. That is 
\[I_\D = (\texttt{x}_{F}~:~F \textrm{~is a minimal nonface of~}\D).\]
We will denote by $K[\D] = S/I_\D$ the \emph{Stanley-Reisner ring} of $\D$. It is a well known fact that $\dim K[\D] = \dim\D +1$. We will denote by  $h(\D) = (h_0,h_1,\ldots,h_s) = h(K[\D])$, the $h$-vector of the graded algebra $K[\D]$. In other words, if $H_{K[\D]}(t)$ is the Hilbert series of $K[\D]$, we have
\[ H_{K[\D]}(t) = \frac{h_0+h_1t + \ldots + h_st^s}{(1-t)^d},\]
where $d$ is the Krull dimension of $K[\D]$ and $h_s\neq 0$. The sequence $h(\D)$ is called the \emph{$h$-vector of $\D$}. The $h$-vector of $\D$ can be determined directly from the $f$-vector of $\D$ using the relation:
\[\sum_{i=0}^{d}f_{i-1}(t-1)^{d-i} = \sum_{i=0}^{d}h_it^{d-i}.\] 
Comparing the coefficients we obtain the formula:
\begin{equation}\label{hvect}
h_j=\sum_{i=0}^j(-1)^{j-i}\binom{d-i}{j-i}f_{i-1}.
\end{equation}
It is  well known  that $s\le d$. So, as opposed to the $f$-vector, the $h$-vector does not contain precise information about the dimension of the simplicial complex. In other words, the $f$-vector can be determined from the $h$-vector only if  the dimension of $\Delta$ is also known.

Let $\D$ be a simplicial complex and $F$ a face of $\D$. The \emph{link of $F$ in $\D$} is the following simplicial complex: 
\[\link_{\D}F = \{ F' \in \D ~:~ F' \cup F \in \D \textrm{~and~} F' \cap F = \emptyset\}.\]
For a set of vertices $W\subset [n]$, the \emph{restriction of $\D$ to $W$} is the following subcomplex of $\D$:
\[ \D_W = \{F\in\D~:~F\subset W \}.\]
The subcomplex $\D_W$ is also  called the subcomplex of $\D$ \emph{induced by} the vertex set $W$.
If $[n]\setminus W = F$ is a face of $\D$, the subcomplex  $\D_W$ is called the  \emph{face deletion} of $F$ in $\D$.
We will abuse notation  and write $\D \setminus F = \{F' \in\D~:~F \not\subset F'\}$ for the face deletion of $F \in \D$. Whenever $F$ is a 0-dimensional face $\{v\}$ we will just write $\D\setminus v$ for the face deletion of $\{v\}$ and $\lk_\D v$ for the link of $\{v\}$.

Consider  $\D' \subseteq \D$ a subcomplex and let $\G$ be a simplicial complex with vertex set disjoint from the vertex set of $\D$. We define the \emph{star of $\D$ with $\G$ along $\D'$} to be the simplicial complex:
\[ \D\ast_{\D'} \G = \D ~\bigcup ~\{F' \cup F ~:~ F' \in \D' \textrm{~and~} F \in \G\}.\] 
It is easy to see that, for any $F\in\D$ the three definitions above are connected in the following way:
\[\D = (\D\setminus F)\ast_{\link_\D F} \langle F\rangle.\]

A simplicial complex $\D$ on $[n]$ is said to be $k$-colorable, for some $k\in \NN$, if there exists a function $\col : [n] \longrightarrow [k]$ such that if $\col(i)=\col(j)$ for $i\neq j$, then no face of $\D$ contains both $i$ and $j$. Obviously, if the dimension of $\D$ is $d-1$, then $k\geq d$. A ($d-1$)-dimensional simplicial complex is called \emph{balanced} if it is $d$-colorable.
For a balanced simplicial complex and for every $i\in [d]$ we  denote by  $V_i=\{v \in [n] ~:~ \col(v)=i\}$ the set of vertices colored $i$.
Fixing a coloring, the Stanley-Reisner ring of a balanced simplicial complex has a canonical linear system of parameters (see \cite[Proposition 4.3]{St}), given by 
\[ \theta_i= \sum_{j\in V_i} x_j.\]

A simplicial complex $\D$ is called Cohen-Macaulay (CM for short) over a field $K$ if and only if the ring $K[\D]$ is Cohen-Macaulay. If $\D$ is CM over any field $K$ then we simply say that $\D$ is CM. There are several combinatorial properties of simplicial complexes that imply Cohen-Macaulayness. In this paper we will focus on the following one.
A pure simplicial complex $\Delta$ is recursively defined to be \emph{vertex decomposable} if it is either a simplex or else has some vertex $v$ such that:
\begin{compactenum}
\item both $\Delta\setminus v$ and $\lk_\Delta v$ are vertex decomposable,
\item no face of $\lk_\Delta v$ is a facet of $\Delta\setminus v$.
\end{compactenum}
A vertex satisfying condition 2. above is called a \emph{shedding vertex}. 
As we mentioned above, a vertex decomposable simplicial complex is always CM. The other implication is known to be false in general. 

\begin{remark}
 If $\Delta$ is  vertex decomposable and balanced, the sets $V_i$ that we defined above are uniquely determined.
\end{remark} 

A simplicial complex is called \emph{flag} if all its minimal nonfaces have cardinality two. In other words, if its Stanley-Reisner ideal is generated by square-free monomials of degree two. Flag simplicial complexes are closely related to simple graphs, i.e. finite graphs with neither loops nor multiple edges. Let $G$ be a (simple) graph on the vertex set $V(G) = [n]$ and let $E(G)$ denote the set of its edges. We define the edge ideal of $G$ as the ideal: 
\[I(G) = (x_ix_j~:~ \{i,j\}\in E(G)) \subset S.\]
For a flag simplicial complex $\Delta$ we will denote by $G_{\Delta}$, or just $G$ if no confusion arrises, the graph of minimal nonfaces of $\Delta$. In particular $I_\Delta = I(G_\D)$.

Given the correspondence between Stanley-Reisner ideals of flag simplicial complexes and edge ideals of simple graphs we  also need to introduce some terminology related to graphs.  For a vertex $v \in V(G)$ we denote by $N(v) = \{w \in V(G)~:~ \{v,w\} \in E(G)\}$ the \emph{open neighborhood} of $v$ in $G$. By $N[v]$ we denote the \emph{closed neighborhood} of $v$, i.e. $N(V) \cup \{v\}$.  For a subset of vertices $W\in V(G)$ we define:
\[N(W) = (\bigcup_{v\in W}N(v))\setminus W.\]
A \emph{perfect matching} of $G$ is a collection of disjoint edges $\{e_1,\ldots,e_r\}$ of $G$ such that every vertex belongs to one of the edges, i.e. $V = \cup\, e_i$.
An $\emph{independent}$ set in $G$ is a collection of vertices $\{v_1,\ldots,v_r\}$ such that $\{v_i,v_j\} \notin E(G)$ for any $i,j \in \{1,\ldots,r\}$. An independent set is called maximal if it is not strictly included in any other independent set of $G$. Notice that the independent sets of $G$ form a simplicial complex, which we will denote by $\D(G)$. It is easy to see that $G_{\D(G)}=G$ and $\D(G_{\D})=\D$. 
A \emph{vertex cover} of $G$ is a collection of vertices $C = \{v_1,\ldots,v_t\}$ such that $e \cap C \neq \emptyset$ for any $e\in E(G)$. A vertex cover is called \emph{minimal} if no proper subset of $C$ is again a vertex cover. The smallest  cardinality of minimal vertex covers of $G$ is called the \emph{covering number } of $G$ and we will denote it by $\tau(G)$.

\begin{lemma}\label{dim=codim}
Let $G$ be a graph without isolated vertices on $[2d]$ such that $\tau(G)=d$. Suppose that any vertex of $G$ belongs to a maximal independent set of cardinality $d$. Then $G$ admits a perfect matching. 
\end{lemma}
\begin{proof}
Let $C=\{v_1,\ldots ,v_d\}\subseteq V(G)$ be a minimal vertex cover of cardinality $d$.  Notice that for any $i=1,\ldots ,  d$ there exists a maximal independent set $H$, of cardinality $d$, such that $v_i \in H$. So there exist $k\leq d$ maximal independent sets $H_1,\ldots ,H_k$ of cardinality $d$, such that 
$$ C\subseteq \bigcup_{j=1}^k H_j.$$ 
Set $F=V(G)\setminus C$. By  definition $F$ is a maximal independent set of $G$ of cardinality $d$. For any $j=1,\ldots ,  k$, set $C_j=C\cap H_j$. Notice that $|F\cap N(C_j)|=|C_j|$ for any $j=1,\ldots ,  k$. In fact, since $H_j$ is a maximal independent set, it is easy to show that $F\cap N(C_j)=F\setminus H_j$, so 
\[|F\cap N(C_j)|=|F\setminus H_j|=|F|-|F\cap H_j|=d-(d-|C_j|)=|C_j|.\]
For any $j=1,\ldots ,  k$, set $\displaystyle A_j = C_j \setminus (\bigcup_{p=1}^{j-1}C_p)$ and  $\displaystyle B_j = (F\cap N(C_j))\setminus (\bigcup_{p=1}^{j-1}(F\cap N(C_p)))$.

{\it Claim 1}. For any $j=1, \ldots ,k$ we have $|A_j|=|B_j|$.\\
Set $\widetilde{C_j}=C_j \cap (\bigcup_{p=1}^{j-1}C_p)$. If we had $|\widetilde{C_j}|<|F\cap N(\widetilde{C_j})|$,
then $(C\setminus \widetilde{C_j})\cup (F\cap N(\widetilde{C_j}))$
would be a vertex cover of cardinality less than $d$. Thus 
\[|\widetilde{C_j}| \geq|F\cap N(\widetilde{C_j})|.\]
Putting everything together we obtain
\[ |B_j| = |F\cap N(C_j)| - |F\cap N(\widetilde{C_j})|\leq |C_j| - |\widetilde{C_j}|= |A_j|. \]
But then $d=\sum_{j=1}^k |B_j| \leq \sum_{j=1}^k |A_j|=d$,
from which we get the claim.

For any $j=1,\ldots ,  k$ let $G^j$ denote the subgraph of $G$ induced by  $\bigcup_{p=1}^j (A_p\cup B_p)$.

{\it Claim 2}. For any $j=1,\ldots ,  k$ the graph $G^j$ has a perfect matching.\\
We will prove Claim 2 by induction. Notice that $G^1$ is a bipartite graph with bipartition 
\[C_1\cup (F\cap N(C_1)).\]
The covering number of $G^1$ is $|C_1|=|F\cap N(C_1)|$. In fact, if $C'$ were a vertex cover of $G^1$ of cardinality less than $|C_1|$, then $C'\cup (C\setminus C_1)$ would be a vertex cover of $G$ of cardinality less than $d$, a contradiction. Therefore $G^1$ has a perfect matching by K\"onig's theorem (\cite[Theorem 1.1.1]{LP}).

Assume that $G^{j-1}$ has a perfect matching. Consider the bipartite subgraph of $G$ induced on the vertices of $C_j \cup (F\cap N(C_j))$. As above, by K\"onig's theorem,  it has a perfect matching. Moreover, such a perfect matching restricts to a perfect matching of the subgraph of $G$ induced by $A_j\cup B_j$, since 
\[F \cap N(\widetilde{C_j})\subseteq B_j.\]
So we can extend the perfect matching of $G^{j-1}$ to a perfect matching of $G^j$.
\end{proof}

Before we state the next theorem we recall a graph theoretical notion from \cite{CV}. An edge $e$ of a graph $G$ is called {\it right edge} if $|C\cap e|=1$ for any minimal vertex cover $C$ of $G$. By the paper of the second author with Benedetti \cite{BV}, $e=\{i,j\}$ is right if and only if 
$\forall~\{i,i'\},\{j,j'\}\in E(G) \Rightarrow \{i',j'\} \in E(G)$.
Finally, recall that $G$ satisfies the {\it weak square condition} if every  vertex of $G$ belongs to a right edge.

\begin{thm}\label{structure}
Let $\Delta=\Delta(G)$ be a $(d-1)$-dimensional flag simplicial complex on $[2d]$ without cone points. The following are equivalent:
\begin{compactenum}
\item $G$ has a perfect matching of right edges, $\{\{u_1,v_1\},\ldots ,\{u_d,v_d\}\}$, such that $\{u_1,\ldots ,u_d\}$ is an independent set and if $\{u_i,v_j\}$ is an edge of $G$ then $i\leq j$. 
\item $\D$ is strongly connected.
\item $\D$ is Cohen-Macaulay over any field.
\item $G$ has a unique perfect matching and it is unmixed.
\item $\D$ is vertex decomposable.
\end{compactenum}
\end{thm}
\begin{proof}The equivalence of the first four points is known from \cite[Theorem 4.7]{CV} for graphs that satisfy the weak square condition. So we only need to check that every vertex of $G$ belongs to a right edge.
Each of the first four properties implies  that $\D$ is pure. In particular any vertex of $G$ belongs to an independent set of cardinality $d$. So by Lemma \ref{dim=codim} $G$ has a perfect matching, say $\{e_1,\ldots ,e_d\}\subseteq E(G)$. Since $\D$ is pure of dimension $d-1$, for any minimal vertex cover $C\subseteq V(G)$ we have $|C\cap e_i|=1$ for any $i=1,\ldots ,  d$. This means that $G$ satisfies the weak square condition, so \cite[Theorem 4.7]{CV} implies that  the properties (1.), (2.), (3.) and (4.) are equivalent. 

Since a vertex decomposable simplicial complex is always CM, (5.) $\Rightarrow$ (3.) follows.
We will argue by induction on $d$ to prove that (1.) $\Rightarrow$ (5.).  If $d=1$ it is trivial, since any $0$-dimensional simplicial complex is vertex decomposable. 

Therefore consider $d\geq 2$. Clearly $v_d$ is a shedding vertex of $\D$, and $\D \setminus v_d$ and $\link_{\D}v_d$ are flag simplicial complexes. Precisely they are $\D \setminus v_d=\D(G_1)$ and $\link_{\D}v_d=\D(G_2)$ where $G_1$ is the subgraph of $G$ induced on the set of vertices $V(G)\setminus \{v_d\}$ and $G_2$ is the subgraph of $G$ induced on the set of vertices $V(G)\setminus N[v_d]$. Notice that $\D \setminus v_d$ is a $(d-1)$-dimensional simplicial complex as well as $\D$. Clearly the graph $G_1^{red}$ obtained from $G_1$ after removing its  (unique) isolated vertex, is a graph on $2(d-1)$ vertices such that (1.) is easily seen holding true. So $\D(G_1^{red})$ is vertex decomposable by induction, and since $\D \setminus v_d$ is obtained from $\D(G_1^{red})$ adding some cone points, it is vertex decomposable too. We want to show that (1.) holds true also for $G_2^{red}$. To see this, assume that $u_i$ is not a vertex of $G_2$ for some $i<d$. Then, using the fact that $\{u_i,v_i\}$ is right, it is easy to see that $v_i$ is an isolated vertex in $G_2$. Analogously, if $v_i$ is not a vertex of $G_2$ then $u_i$ is an isolated vertex of $G_2$. Hence the perfect matching of $G$ induces  a perfect matching on $G_2^{red}$. At this point it is easy to see that (1.) holds true for $G_2^{red}$, so using the above argument $\link_{\D}v_d$ is vertex decomposable by induction. Therefore $\D$ is vertex decomposable. 
\end{proof}

We conclude this section with a useful remark.
Let $A=S/J$ an Artinian $K$-algebra. We will say that $A$ is a {\it quadratic Artinian $K$-algebra} if $J$ is generated by quadrics, and that $A$ is a {\it monomial Artinian $K$-algebra} if $J$ is generated by monomials.  
\begin{remark}\label{artinian}
Let $\D$ be a simplicial complex on $[n]$. Construct the ideal
\[J_{\D}=I_{\D}+(x_1^2,\ldots ,x_n^2)\subseteq S.\]
It is straightforward to verify that $S/J_{\D}$ is a {\it monomial Artinian $K$-algebra} such that
\[h(S/J_{\D})=f(\D).\]
On the other hand, if $A=S/J$ is a monomial Artinian $K$-algebra such that $x_i^2\in J$ for any $i=1,\ldots ,  n$, then $J=I_{\D}+(x_1^2,\ldots ,x_n^2)$ for some simplicial complex $\D$ on $[n]$. Once again we have
\[h(A)=f(\D).\]
Therefore the set of  $h$-vectors of monomial Artinian $K$-algebras whose defining  ideal contains the square of each variable is equal to the set of $f$-vectors of simplicial complexes.
By the same argument,  characterizing the $f$-vectors of flag simplicial complexes is  equivalent to classifying the $h$-vectors of quadratic monomial Artinian $K$-algebras.
\end{remark}


\section{$h$-vectors of Vertex Decomposable Flag Simplicial Complexes}
In this section we are going to prove  Conjecture \ref{weakkalai}  when $\D$ is vertex decomposable. 
First of all, we want to remark that the inclusion in Conjecture \ref{EGHconj} is strict. To this aim let us take a look at the next example.
\begin{example}\label{triangle}
Consider the $f$-vector of the empty  triangle, $(1,3,3)$. If  a quadratic Artinian $K$-algebra with $h$-vector $(1,3,3)$ existed, then it would be of the kind: 
\[A=K[x,y,z]/(f_1,f_2,f_3),\]
where the $f_i$'s are degree $2$ homogeneous polynomials of $K[x,y,z]$. Since the ideal $(f_1,f_2,f_3)$ is a complete intersection, the $h$-vector of $A$ has to be symmetric -- a contradiction.
\end{example}

Before stating the main result of this section we will prove the following algebraic lemma.

\begin{lemma}\label{inequality $h$-vectors}
Let $A$ be a standard graded Noetherian $d$-dimensional Cohen-Macaulay $K$-algebra and $J\subseteq A$ a height $1$ ideal generated by elements of degree $1$ such that $A/J$ is Cohen-Macaulay. If $K$ is infinite, then for any $i\in \NN$
\[h_i(A/J)\leq h_i(A).\]  
\end{lemma}
\begin{proof}
By \cite[Proposition 1.5.12]{BH} we can choose  a degree $1$ homogeneous element $x\in J$ which is $A$-regular. Thus for any $i$ we have that $h_i(A/(x))=h_i(A)$. Moreover $A/(x)$ and $A/J$ have the same dimension.  Let us extend $x$ to a regular sequence for $A$ of degree $1$ elements, say $x,x_2,\ldots ,x_d$ where $d=\dim(A)$. It turns out that $x_2,\ldots ,x_d$ is a system of parameters for $A/J$. Because $A/J$ is Cohen-Macaulay,  $x_2,\ldots ,x_d$ is a regular sequence for $A/J$. So there is a graded surjection 
\[A/(x,x_2,\ldots ,x_d) \longrightarrow A/(J+(x_2,\ldots ,x_d)),\]
from which we get the desired inequality: 
\[ h_i(A/J) \leq h_i(A).\]
\end{proof}

We are ready to prove the main theorem of this section.

\begin{thm}\label{generalvd}
Let $\Delta$ be a vertex decomposable, flag simplicial complex. Then there exists a simplicial complex $\Gamma$ such that $f(\Gamma)=h(\Delta)$.
\end{thm}
\begin{proof}
Suppose that $\D$ is $d$-dimensional on $[n]$.
If $\Delta$ is the $d$-simplex, then it is enough to choose  $\Gamma = \{\emptyset\}$.
 So we can assume that $\Delta$ is not a simplex  and use induction on $d$ and $n$.

Let $v$ be a shedding vertex of $\Delta$ such that $\Delta_1=\Delta \setminus \{v\}$ and $\Delta_2=\link_{\Delta}v$ are vertex decomposable simplicial complexes. We may assume $v=n$, so it turns out that $\Delta_1$ is of dimension $d$ on $[n-1]$, whereas $\Delta_2$ is of dimension $d-1$. For any $i=0,\ldots ,d$ we have
\[f_i(\Delta)=|\{\mbox{$i$-faces of $\Delta$ not containing $v$}\}|+|\{\mbox{$i$-faces of $\Delta$ containing $v$}\}|=f_i(\Delta_1)+f_{i-1}(\Delta_2).\]
Using \eqref{hvect} it is not difficult to show that the same formula holds at the $h$-vectors' level:
\[ h_i(\Delta)=h_i(\Delta_1)+h_{i-1}(\Delta_2) \mbox{ \ \ for every \ }i=1,\ldots ,  (d+1). \]
Before proceeding with the induction we will prove the following:

{\it Claim. For any $i$ we have $h_i(\Delta_2)\leq h_i(\Delta_1)$.} \\
By definition we have that 
\[ I_{\Delta_1}=(x_{i_1}x_{i_2} \ :  \ \{i_1,i_2\} \notin\Delta \mbox{ and }v\notin \{i_1,i_2\}), \]
\[I_{\Delta_2}=(x_{i_1}x_{i_2} \ :  \ \{i_1,i_2\} \notin\Delta , \ v\notin \{i_1,i_2\} \mbox{ and both $\{i_1,v\}, \{i_2,v\} \in \Delta$}).\]
Moreover $K[\Delta_1]=K[x_i : i\neq v]/I_{\Delta_1}$ and $K[\Delta_2]=K[x_i : i\neq v \mbox{ and }\{i,v\} \in\Delta]/I_{\Delta_2}$. Therefore
\[K[\Delta_2]=K[\Delta_1]/(x_i \ : \ \{i,v\} \notin \Delta).\]
Since  $\Delta_1$ and $\D_2$ are  vertex decomposable, $K[\Delta_1]$ and $K[\D_2]$ are Cohen-Macaulay. So we are in the situation of Lemma \ref{inequality $h$-vectors}. Hence
\[h_i(\Delta_2)=h_i(K[\Delta_2])\leq h_i(K[\Delta_1])=h_i(\Delta_1),\]
and the claim follows.

By induction there exist two simplicial complexes, $\Gamma_1$ and $\Gamma_2$, such that $f(\Gamma_1)=h(\Delta_1)$ and $f(\Gamma_2)=h(\Delta_2)$. We want to construct the desired simplicial complex $\Gamma$ starting from them. 
By the Kruskal-Katona theorem (for instance see \cite[Theorem 2.1]{St}) we can assume that both $\Gamma_1$ and $\Gamma_2$ are rev-lex complexes. Therefore, since by the claim $f_i(\Gamma_2)\leq f_i(\Gamma_1)$, actually $\Gamma_2$ is a subcomplex of $\Gamma_1$. So it makes sense to construct the simplicial complex 
\[\Gamma = \Gamma_1 *_{\Gamma_2} \{u\},\]
where $u$ is a new vertex. It is straightforward to check that 
\[f_i(\Gamma)=f_i(\Gamma_1)+f_{i-1}(\Gamma_2)=h_i(\Delta_1)+h_{i-1}(\Delta_2)=h_i(\Delta).\]
\end{proof}

The reader might think at this point that $h$-vectors of quadratic Artinian $K$-algebras are $h$-vectors of Cohen-Macaulay flag simplicial complexes. The following example will show that this is not the case.

\begin{example}\label{balanced}
Let $h=(1,4,5,1)$ be a sequence of integers (notice that $h$ is the $f$-vector of a balanced simplicial complex). 
In the paper of Roos \cite{Ro} we found  the quadratic Artinian $K$- algebra $A=K[x_1,x_2,x_3,x_4]/I$, where $I$ is the ideal
\[I=(x_1x_2+x_3^2, \ x_1x_4, \ x_1^2+x_3^2+x_4^2, \ x_2^2, \ x_2x_3+x_3x_4),\]
with  $h(A)=(1,4,5,1)$.

If there existed a Cohen-Macaulay flag simplicial complex $\D$ with $h(\D)=h$, then there would exist an Artinian Koszul $K$-algebra $B$ with $h(B)=h$. In fact, if $\theta=\theta_1,\ldots ,\theta_d$ is a system of parameters for $K[\D]$, it is enough to take  $B=K[\D]/(\theta)$. This follows from  the theorem of Fr\"oberg \cite{Fr1} and the result of Backelin and Fr\"oberg \cite[Theorem 4]{BF}. This implies that
\[\displaystyle \frac{1}{1-4z+5z^2-z^3}=\sum_{i\geq 0}\dim_K(\Tor_i^B(K,K))z^i,\]
(for instance see \cite[p. 87]{BF}). Computing the coefficients on the left hand side we obtain $\dim_K(\Tor_9^B(K,K))=-174$, obviously a contradiction.
\end{example}

In  light of Examples \ref{triangle} and \ref{balanced}, we conclude this section discussing whether the simplicial complex $\Gamma$ of Theorem \ref{generalvd} could be chosen with some extra properties. First of all we have a remark.
\begin{remark}\label{alternativeflag}
It is easy to see that the following holds true: A simplicial complex $\D$ on $[n]$ is flag if and only if $\D = \{\emptyset\}$ or there exists a vertex $v$ of $\D$ such that $\D\setminus v$ is flag and $\lk_{\D}v=\D_{W}$ for some $W\subseteq [n]$ (in particular $\lk_{\D}v$ is flag). 
\end{remark}
Let $\Gamma_2 \subseteq \Gamma_1$ be two  simplicial complexes, with $\G_2 =(\Gamma_1)_W$ induced by a subset of vertices $W\subseteq [n]$. Then
\[\displaystyle \frac{K[x_i \ : \ i\in W]}{J_{\Gamma_2}}\cong \frac{S}{J_{\Gamma_1}+(x_j \ : \ j\notin W)},\]
where $J_{\Gamma_1}$ and $J_{\Gamma_2}$ are the ideals defined in Remark \ref{artinian}. Therefore
\[\displaystyle f(\Gamma_2)=h\left(\frac{S}{J_{\Gamma_1}+(x_j \ : \ j\notin W)}\right).\]
Thus we are in the situation in which there exists a monomial Artinian  $K$-algebra $A$ and an ideal $I\subseteq A$ generated by variables such that
\[f(\Gamma_1)=h(A) \mbox{ \ and \ }f(\Gamma_2)=h(A/I).\]
Moreover, if $A$ is quadratic, then by Remark \ref{alternativeflag} the complex $\Gamma_1*_{\Gamma_2}\{v\}$ is  flag.
In the proof of Theorem \ref{generalvd} we have that $K[\Delta_2]=K[\Delta_1]/I$, where $I$ is an ideal generated by variables. 
Since $K[\D_1]$ and $K[\D_2]$ are both Cohen-Macaulay, going modulo a generic regular sequence,
we could  restrict  to the Artinian case.
The problem is that the quadratic Artinian reduction $A$ of $K[\D_1]$ is not necessarily  monomial. This is why, even assuming that $\Gamma_1$ and $\Gamma_2$ are flag, we could not  conclude that  $\Gamma_1*_{\Gamma_2}\{v\}$ is also flag. In other words,  if in the proof of Theorem \ref{generalvd} we assume by induction that  $\Gamma_1$ and $\Gamma_2$ are flag, we  do not see how to construct a flag simplicial complex $\Gamma$,  because $\Gamma_2$ might not be a subcomplex of $\Gamma_1$ induced by some set of vertices.
However,  the behavior of the $f$-vector of $\Gamma_2$ is similar to that of the $f$-vector of an induced  subcomplex of  $\Gamma_1$. For instance, if $f_0(\Gamma_2)=f_0(\Gamma_1)$, it follows by the proof of Theorem \ref{generalvd} that $f_i(\Gamma_2)=f_i(\Gamma_1)$ for any $i$. In the next section we present more precise results in this direction under the assumption that $\D$ is also balanced (see Definition \ref{quasi-flag} and Theorem \ref{balanced-quasi-flag}). 

\section{Balanced, Vertex Decomposable,  Flag  Complexes}

The reason for which we study balanced, vertex decomposable, flag simplicial complexes comes from the  Proposition \ref{fflaghbvd}. We conjecture that the converse of this proposition is true.
In Theorem \ref{balanced-quasi-flag} we will prove a weaker version of the equality in  Conjecture \ref{conjalexvabba}. Finally we will prove that the conjecture holds for balanced, vertex decomposable, flag $(d-1)$-dimensional simplicial complexes on $[2d]$, without cone points.

\begin{prop}\label{fflaghbvd}
Let $\Gamma$ be a flag simplicial complex. Then there exists a balanced, vertex decomposable, flag simplicial complex $\D$ such that $h(\D)=f(\Gamma)$.
\end{prop}
\begin{proof}
Set $n=\dim \Gamma +1$ and as in Remark \ref{artinian} consider the ideal 
\[J_{\Gamma}=I_{\Gamma}+(x_1^2,\ldots ,x_n^2)\subseteq S.\]
Consider the polarization of $J_{\Gamma}$:
\[J_{\Gamma}^{pol}=I_{\Gamma}+(x_1y_1,\ldots ,x_ny_n)\subseteq P=K[x_1,\ldots ,x_n,y_1,\ldots ,y_n].\]
Since polarization is a particular distraction, it preserves the height and the graded Betti numbers (see the paper of Bigatti, Conca and Robbiano \cite{BCR}). Particularly 
\[h(P/J_{\Gamma}^{pol})=h(S/J_{\Gamma})=f(\Gamma),\]
where the last equality follows from Remark \ref{artinian}. The simplicial complex $\D$ associated to $J_{\Gamma}^{pol}$ is flag. More precisely $\D=\D(G)$ where $G$ is the graph on the vertices $\{u_1,\ldots ,u_n,v_1,\ldots ,v_n\}$ whose edges are $\{u_i,v_i\}$ for $i=1,\ldots ,  n$ and $\{v_i,v_j\}$ such that $\{i,j\}$ is not a face of $\Gamma$. Then by Theorem \ref{structure} $\D$ is vertex decomposable. Moreover $\D$ is easily seen to be balanced setting $\col(u_i)=\col(v_i)=i$ for any $i=1,\ldots ,  n$.
\end{proof}

We conjecture that the converse of Proposition \ref{fflaghbvd} is also true:

\vspace{2mm}

\noindent {\bf Conjecture} \ref{conjalexvabba}. The following equality holds:
\begin{equation}
\!\! \left \{ \! \!
	\begin{array}{c}
	\hbox{$h$-vectors of vertex decomposable} \\
	\hbox{balanced and flag simplicial complexes}
	\end{array}
\right \}
=
\left \{ 
	\begin{array}{c}
	\hbox{$f$-vectors of}  \\
 	\hbox{flag simplicial complexes}
	\end{array}
\!\!\!\right \}. \nonumber
\end{equation}
 Next, we are going to prove a result in support of the above conjecture. This next theorem will be a version of Conjecture \ref{conjalexvabba},  in which we will prove that the hard inclusion ($\subseteq$) holds with weakened conditions on the right hand side of the equality. In Theorem \ref{dimcodimb} and in the two lemmas of Section 6  we will prove that equality holds when adding some stronger conditions on the left hand side.
First we need to define a new class of simplicial complexes, suggested  by Remark \ref{alternativeflag}.

\begin{definition}\label{quasi-flag}
Let $\D$ be a simplicial complex on $[n]$. Then $\D$ is \emph{ quasi-flag} if and only if $n=0$ or there exists a vertex $v$ of $\D$ such that 
\begin{compactenum}
\item$\D \setminus v$ has the $f$-vector of a quasi-flag simplicial complex, 
\item$\lk_{\D}v=\D_W$ for some $W\subseteq [n]$ and the $f$-vector of $\lk_{\D}v$ is that of a quasi-flag simplicial complex.
\end{compactenum}
\end{definition}
\begin{thm}\label{balanced-quasi-flag}
Let $\Delta$ be a balanced, vertex decomposable, flag simplicial complex on $[n]$. Then there exists a quasi-flag simplicial complex $\Gamma$ such that $f(\Gamma)=h(\Delta)$.
\end{thm}
\begin{proof}
If $\D$ is a simplex then we can choose  $\Gamma = \{\emptyset\}$. If $\D$ is not a simplex we can choose a shedding vertex $v$ such  that $\Delta_1=\Delta \setminus \{v\}$ and $\Delta_2=\link_{\Delta}v$ are vertex decomposable, flag simplicial complexes. As in the proof of Theorem \ref{generalvd}, we have
\[K[\Delta_2]=K[\Delta_1]/(x_i:i\in W),\]
where $W=\{i:\{i,v\} \notin \Delta\}$. Let $\col:[n]\rightarrow [d]$ be a $d$-coloring of $\D$, where $\dim \D = d-1$. For any $j=1,\ldots ,  d$ we set $V_j=\{i\in [n]:\col(i)=j\}$. We can assume that $v=n\in V_d$. Notice that the coloring on $\D$ induces a $d$-coloring on $\D_1$ and a $(d-1)$-coloring on $\D_2$, so that $\D_1$ and $\D_2$ are both balanced. So we have the following  system of parameters for $K[\D_1]$:
\[\theta_i=\sum_{{j\in V_i}\atop{j\neq n}}x_j, \ \ \ \ i=1,\ldots ,  d.\]
It turns out that $\theta_i$, where $i=1,\ldots ,  (d-1)$, provides also a system of parameters for $K[\Delta_2]$. Note that $\theta_d$ is zero in $K[\D_2]$. We may assume that $i\in V_i$ for any $i=1,\ldots ,  d$ and that $i\notin W$ for any $i=1,\ldots ,  (d-1)$. Consider the ideal of $K[x_{d+1},\ldots ,x_{n-1}]$:
\[I=(x_ix_j, \ x_i(\sum_{{k\in V_h}\atop{k\neq h}}x_k) \ : \ d+1\leq i,j\leq n-1, \ h=1,\ldots ,  d, \mbox{ and } \{i,j\}, \{i,h\}\notin \Delta_1 ).\]
Going modulo the $\theta_i$'s, it is easy to see that
\[\displaystyle \frac{K[\Delta_1]}{(\theta_1,\ldots ,\theta_d)}\cong \frac{K[x_{d+1},\ldots ,x_{n-1}]}{I}=A.\]
Moreover
\[\displaystyle \frac{K[\Delta_2]}{(\theta_1,\ldots ,\theta_d)} = \frac{K[\Delta_2]}{(\theta_1,\ldots ,\theta_{d-1})} \cong \frac{A}{(x_i:i\in W)}=B.\]
Since $\D_1$ and $\D_2$ are both Cohen-Macaulay, 
\[h(\D_1)=h(A) \mbox{ \ and \ }h(\D_2)=h(B).\]
Notice that $x_i^2\in I$ for any $i=d+1,\ldots ,n-1$. So for any term-order $\prec$ in $K[x_{d+1},\ldots ,x_{n-1}]$ there exists a simplicial complex $\Gamma_1$ such that
\[\LT(I)=J_{\Gamma_1}.\]
If we consider as $\prec$ a deg-rev-lex term-order such that the smallest variables are the $x_i$'s with $i\in W$,  we have
\[\LT(I+(x_i:i\in W))=J_{\Gamma_1}+(x_i:i\in W),\]
see for instance the book of Eisenbud \cite[Proposition 15.12]{Ei}.
By the above discussion we have $f(\Gamma_1)=h(\D_1)$ and $f((\Gamma_1)_W)=h(\D_2)$. By induction $\Gamma_1$ and $\Gamma_2=(\Gamma_1)_W$ have both the $f$-vector of quasi-flag simplicial complexes. So
\[\Gamma = \Gamma_1 *_{\Gamma_2} \{u\},\]
where $u$ is a new vertex, is a quasi-flag simplicial complex. As in the proof of Theorem \ref{generalvd} we have $f(\Gamma)=h(\D)$, thus we conclude (notice that since $\Gamma_2$ is already contained in $\Gamma_1$ this time we need not use the Kruskal-Katona theorem).
\end{proof}
\begin{remark}
By Remark \ref{alternativeflag} the flag simplicial complexes are quasi-flag. 
However notice that not all the $f$-vectors are $f$-vectors of quasi-flag simplicial complexes. For instance take $f=(1,n,\binom{n}{2})$. The unique simplicial complex with such an $f$-vector is the complete graph $K_n$. However the link of any vertex of $K_n$ is not a subcomplex of $K_n$ induced by a set of vertices. Thus $K_n$ is not quasi-flag.

Another example is also provided by the $f$-vector $(1,4,5,1)$. The unique simplicial complex $\D$ which has such an $f$-vector is the one whose set of facets is
\[\mathcal{F}(\D)=\{\{1,2\},\{1,3\},\{2,3,4\}\}.\]
The unique vertex $v$ such that $\lk_{\D}v$ is an induced subcomplex of $\D$ is  $4$. However the $f$-vector of $\D \setminus 4$ is $(1,3,3)$, which is not the $f$-vector of a quasi-flag simplicial complex by the above considerations. Therefore $\D$ is not quasi-flag.

We are not aware of any example of quasi-flag simplicial complex whose $f$-vector is not flag.
\end{remark}

Some evidence in favor of Conjecture \ref{conjalexvabba} is also provided by the following theorem.
\begin{thm}\label{dimcodimb}
The following equality holds true:
\begin{equation}
\!\! \left \{ \! \!
	\begin{array}{c}
	h(\D) \ : \ \hbox{$\D$ is a $(d-1)$-dimensional} \\
	\hbox{balanced, vertex decomposable, flag}  \\
 	\hbox{simplicial complex on $[2d]$, without cone points}
	\end{array}
\right \}
=
\left \{ 
	\begin{array}{c}
	f(\Gamma) \ : \ \Gamma \hbox{ is a flag simplicial}  \\
 	\hbox{complex on $[d]$}
	\end{array}
\!\!\!\right \} .\nonumber
\end{equation}
\end{thm}
\begin{proof}
It is easy to see that the proof of Proposition \ref{fflaghbvd} yields that the set on the right hand side is a subset of the one on the left. For the other inclusion
let $\{\{u_1,v_1\},\ldots ,\{u_d,v_d\}\}$ be the perfect matching of $G=G_{\D}$ described in point (1.) of Theorem \ref{structure}. Also denote by
\[P=K[x_1,\ldots ,x_d,y_1,\ldots ,y_d]\] 
 the polynomial ring containing $I_{\D}$, where $x_i$ is the variable associated to $u_i$ and $y_i$ the one associated to $v_i$. 
Notice that  $\D$ is balanced, so by \cite[Proposition 4.3]{St} the set
$$\{\theta_i=x_i+y_i \ : \ i=1,\ldots ,  d\}$$ 
is a system of parameters for $K[\D].$
Thus we have  the graded isomorphism
\[\displaystyle \frac{K[\D]}{(\theta_1,\ldots ,\theta_d)}\longrightarrow \frac{K[z_1,\ldots ,z_d]}{(z_i^2, \ z_hz_k \ : \ i=1,\ldots ,  d, \ \{v_h,v_k\} \mbox{ or } \{u_h,v_k\} \mbox{ is an edge})}\]
which maps $y_i$ to $z_i$ and $x_i$ to $-z_i$. Since $\D$ is Cohen-Macaulay over $K$ we have
\[\displaystyle h\left( \frac{K[\D]}{(\theta_1,\ldots ,\theta_d)}\right)=h(\D).\]
So, by the above graded isomorphism, we have
\[h(\D)=\displaystyle h\left(\frac{K[z_1,\ldots ,z_d]}{(z_i^2,z_hz_k \ : \ i=1,\ldots ,  d, \ \{v_h,v_k\} \mbox{ or } \{u_h,v_k\} \mbox{ is an edge})}\right).\]
Using Remark \ref{artinian} we obtain the desired conclusion.
\end{proof}

\section{$h$-vectors of Cohen-Macaulay  Flag Complexes}

In this section we are going to discuss a natural generalization of Conjecture  \ref{conjalexvabba}, namely:

\vspace{2mm}

\noindent {\bf Conjecture} \ref{conjalexvabba1}
The following equality holds true:
\begin{equation}
\!\! \left \{ \! \!
	\begin{array}{c}
	\hbox{$h$-vectors of Cohen-Macaulay,} \\
	\hbox{ flag simplicial complexes}
	\end{array}
\right \}
=
\left \{ 
	\begin{array}{c}
	\hbox{$f$-vectors of}  \\
 	\hbox{flag simplicial complexes}
	\end{array}
\!\!\!\right \}. \nonumber
\end{equation}
One reason for the above conjecture is given by the following remark.
\begin{remark}
 Conjecture \ref{conjalexvabba1} implies Kalai's Conjecture \ref{kalai}. 
\end{remark}
\begin{proof}If $\D$ is a $d$-dimensional, CM, flag simplicial complex then, if Conjecture \ref{conjalexvabba1} were true, there would exist a $s$-dimensional,  flag simplicial complex $\Gamma'$ with $f(\Gamma')=h(\D)$, where  $s\leq d$. By \cite{Fr} there exists also  a \mbox{$s$-dimensional}, balanced simplicial complex $\Gamma''$, with $f(\Gamma'')=f(\Gamma')$. By a result of Bj\"orner, Frankl and Stanley \cite[Theorem 1]{BFS}, there exists a $s$-dimensional CM, balanced simplicial complex $\Gamma'''$ with $h(\Gamma''')=f(\Gamma'')$. Thus $h(\Gamma''')=h(\D)$. Adding $d-s$ cone points to $\Gamma'''$ we get a $d$-dimensional simplicial complex $\Gamma$ that is still CM and balanced. Furthermore $h(\Gamma)=h(\Gamma''')=h(\D)$. Since $\dim \Gamma = \dim \D$, we get $f(\Gamma)=f(\D)$. 
\end{proof}

The set on the right hand side of the equality in Conjecture \ref{conjalexvabba1} is contained in the one on the left by Proposition \ref{fflaghbvd}. So the hard part of the conjecture  is  to prove that  for any Cohen-Macaulay, flag simplicial complex $\D$ there exists a flag simplicial complex $\Gamma$ with $f(\Gamma)=h(\D)$. 

First of all,  notice that as an easy consequence of a more general theorem of Conca, Trung and Valla (\cite{CTV}), we obtain the validity of Conjecture \ref{conjalexvabba1} when the $h$-vector of $\D$ is \lq\lq short enough\rq\rq. Here is the precise statement:

\begin{prop}\label{short}
Let $\D$ be a Cohen-Macaulay, flag simplicial complex with $h$-vector $(1,n,m)$. Then there exists a flag simplicial complex $\Gamma$ with $f(\Gamma)=h(\D)$.
\end{prop} 
\begin{proof}
The $K$-algebra $K[\D]$ is Koszul by  \cite{Fr1}. Taking a regular sequence $\theta_1,\ldots ,\theta_d$, where $d-1=\dim \D$, we get that $A=K[\D]/(\theta_1,\ldots ,\theta_d)$ is a Koszul Artinian $K$-algebra by \cite[Theorem 4]{BF}. Since $h(A)=h(\D)=(1,n,m)$, we have $m\leq n^2/4$ by \cite[Theorem 3.1]{CTV}. Under this condition it is easy to construct a bipartite graph with $n$ vertices and $m$ edges. Such a bipartite graph can also be seen as  a $1$-dimensional, flag simplicial complex with $f$-vector $(1,n,m)$. 
\end{proof}
In particular the above proposition implies that Conjecture \ref{conjalexvabba1} is true when the dimension of $\D$ is $1$.
The following theorem brings more evidence in favor of Conjecture \ref{conjalexvabba1}.

\begin{thm}\label{dimcodim}
The following equality holds true:
\begin{equation}
\!\! \left \{ \! \!
	\begin{array}{c}
	h(\D) \ : \ \hbox{$\D$ is a $(d-1)$-dimensional, CM, flag} \\
 	\hbox{simplicial complex on $[2d]$, without cone points}
	\end{array}
\right \}
=
\left \{ 
	\begin{array}{c}
	f(\Gamma) \ : \ \Gamma \hbox{ is a flag }  \\
 	\hbox{simplicial complex on $[d]$ }
	\end{array}
\!\!\!\right \} .\nonumber
\end{equation}
\end{thm}
\begin{proof}
If $\D$ is a $(d-1)$-dimensional, CM, flag simplicial complex on $[2d]$ without cone points then $\D$ is vertex decomposable and balanced by Theorem \ref{structure}. Thus Theorem \ref{dimcodimb} yields the conclusion.
\end{proof}
Suppose $\D$ is a CM, flag  simplicial complex, without cone points and $G_\D$ is bipartite with partition of the vertex set $A\cup B$. As both $A$ and $B$ are minimal vertex covers, by the purity of $\D$ we have $|A| = |B|$. 
This implies the following corollary of the above theorem.
\begin{corollary} The following inclusion holds true:
\begin{equation}
\!\! \left \{ \! \!
	\begin{array}{c}
	h(\D) \ : \ \hbox{$\D$ CM, flag simplicial complex,} \\
 	\hbox{with $G_\D$  bipartite}
	\end{array}
\right \}
\subseteq
\left \{ 
	\begin{array}{c}
	f(\Gamma) \ : \ \Gamma \hbox{ is a flag }  \\
 	\hbox{simplicial complex}
	\end{array}
\!\!\!\right \} .\nonumber
\end{equation}
\end{corollary}
We conclude this section with the following remark.
\begin{remark}
If $\D=\D(G)$ is a flag, CM simplicial complex, then $\D(G)$ is pure. In particular any vertex of $G$ belongs to an independent set of cardinality $\dim \D + 1$. Therefore, if $\D$ is a $(d-1)$-dimensional, flag, CM simplicial complex on $[n]$ without cone points, then $n\geq 2d$ by the result of Gitler and Valencia \cite[Theorem 2.1]{GV}. 
\end{remark}
In the spirit of the previous remark, Theorem \ref{dimcodim} can be seen as  the first step towards proving Conjecture \ref{conjalexvabba1}.


\section{Further Results and Examples}
 In this last section we will present two rather technical properties of flag simplicial complexes. We will show that the first property (which we  call \emph{balanced cone-face property} - \eqref{cfp}) implies Cohen-Macaulayness (Proposition \ref{cfCM}) and that the $h$-vector of such a simplicial complex is the $f$-vector of a flag simplicial complex (Lemma \ref{coneface}). For simplicial complexes with the second property \eqref{NP} we will construct a new  complex, with the same $h$-vector, which will satisfy  the hypothesis of Theorem \ref{dimcodimb}. We will also present examples of simplicial complexes with and without these properties.

\begin{lemma}\label{coneface}
Suppose $\D$ is a balanced, flag simplicial complex of dimension $d-1$ and $F_0 = \{a_1,\ldots, a_d\}$  be a facet of $\D$ with the property that: 
\begin{equation}\label{cfp}
\forall ~v \in V(\D), \exists ~1\le i \le d \textup{ such that }(F_0\setminus\{a_i\})\cup\{v\} \textup{ is a facet of } \D.
\end{equation}
Then we have $h(\D) = f(\D\setminus F_0)$.
\end{lemma}
\begin{proof}
For simplicity, we will denote by $(h_0,\ldots,h_r)$ and $(f_{-1},\ldots,f_{d-1})$ the $h$-vector, respectively
 the $f$-vector, of $\D$. The $f$-vector of  $\D\setminus F_0$ will be denoted by $(f'_{-1},\ldots,f'_s)$.
We will prove the lemma by induction. First of all clearly $h_0 = f'_{-1} = 1$ and $h_1 = f'_0 = n-d$. Suppose that we already have $h_{j} = f'_{j-1}$ for all $j\le i$.

The following observation is the key of the proof. As $\D$ is flag, for any $d\ge i>j$,  if $\{v_1,\ldots,v_{i-j}\}$ and $\{w_1,\ldots,w_j\}$ are two faces of $\D$   such that $\{v_k,w_l\}\in\D$ for any $k$ and $l$, then $\{v_1,\ldots,v_{i-j},w_1,\ldots,w_j\}\in\D$.

Every $i$-dimensional  face $F\in\D$ is  a disjoint union: $(F\setminus F_0) \cup (F\cap F_0)$. We will count the $i$-faces of $\D$ with $|F\setminus F_0| =j$.
As $\D $ is balanced, the number of vertices of $F_0$ that are colored different from all the vertices of $F\setminus F_0$ is exactly $d-j$. Choose an $(i-j)$-face $G \subset F_0$   supported on these vertices. It is easy to notice that, by our hypothesis and the above observation, $G\cup(F\setminus F_0) \in \Delta$. As there are $\binom{d-j}{i+1-j}$ different ways to choose $G$, we get that the number of $i$-faces of $\Delta$ with $|F\setminus F_0| =j$ is 
\[ f'_j\cdot\binom{d-j}{i+1-j}.\]
Decomposing the set of $i$-faces of $\Delta$ according to the cardinality of $F\setminus F_0$, we obtain
\[f_{i} = \sum_{j=0}^{i+1} \binom{d-j}{i+1-j}f'_{j-1}.\]
As the the $f$-vector of $\Delta$ can be computed from the $h$-vector of $\Delta$ by the formula:
\[f_i = \sum_{j=0}^{i+1} \binom{d-j}{i+1-j}h_{j},\]
we obtain by the inductive hypothesis that $h_{i+1} = f'_{i}$. 
\end{proof}
Notice we did not request in  Lemma \ref{coneface} that $\D$ is Cohen-Macaulay. This is because, under the hypothesis of the above lemma, $\D$ is always CM.
\begin{prop}\label{cfCM}
If $\D$ is a simplicial complex with the same properties as in the statement of Lemma \ref{coneface} then $\D$ is Cohen-Macaulay.
\end{prop}
 
\begin{proof}
As we have seen in the preliminaries section, a balanced simplicial complex has a canonical linear  system of parameters, namely  $\{\theta_i = \sum_{\col(j) = i} x_j ~:~ i=1,\ldots ,  d\}.$ It is easy to see that the property \eqref{cfp} is equivalent to 
\[ x_{a_i}x_{v} \in \Gens(I_\D) \Rightarrow \col(a_i)=\col(v),\quad \forall ~i=1,\ldots ,  d.\]
Notice also that if $V_i $ is the set of vertices of color $i$, then $x_vx_w \in \Gens(I_\D)$ for  any $ v,w \in V_i$ and $\forall~ i = 1,\ldots,d$.
If we denote by $W = [n]\setminus F_0$, considering the above observation, it is not difficult to see that 
\[ \frac{K[\D] }{(\theta_1,\ldots,\theta_d)} \simeq \frac{K[x_i~:~ i\in W]}{( x_i^2, x_ix_j ~:~ \{i,j\} \textrm{~ minimal nonface of~}\D_W)}.\] 
The isomorphism is obtained by sending $x_i\mapsto x_i$ if $i\notin F_0$ and 
$$x_i \mapsto -\sum_{\col(j)=\col(i)}x_j\quad  \textrm{if~~} i \in F_0.$$
By Remark \ref{artinian} we obtain that 
\[h\left(\frac{K[\D]}{(\theta_1,\ldots,\theta_d)}\right) = f(\D_W).\]
As $\D_W = \D\setminus F_0$, by Lemma \ref{coneface} we also have  that 
\[h\left(\frac{K[\D]}{(\theta_1,\ldots,\theta_d)}\right) = h(K[\D]),\]
which by \cite[Lemma 2.6]{St} implies that $\D$ is Cohen-Macaulay.
\end{proof}

Let us present now an example of a simplicial complex satisfying \eqref{cfp}. First let us establish a graphical convention.
Throughout  this section,   the thicker vertical lines in the  pictures of graphs  represent the fact that the subgraphs induced by the vertices in one column  are complete (e.g. in the next figure, the subgraphs induced by each of the vertex sets $\{1,4,7\}$, $\{2,5,8\}$ and $\{3,6,9\}$ are complete).

\begin{example} The independence complex $\D$ of the graph on the left is an example of simplicial complex satisfying the hypothesis of Lemma \ref{coneface}. It is easy to see that $F_0 = \{1,2,3\}$ satisfies property \eqref{cfp}. One can check that $\D$ is pure, of dimension 2 and that $h(\D) = (1,6,5)$.

\includegraphics[scale=0.7]{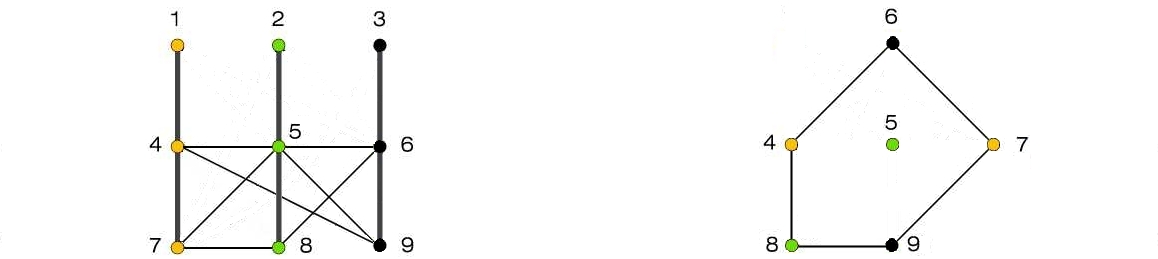}\\
On the right hand side you can see a picture of the 1-dimensional simplicial complex $\D \setminus F_0$. One can notice that $\D \setminus F_0$ is no longer pure, nor balanced. The only property inherited from $\D$, apart from flagness, is the 3-colorability. 
\end{example}

In the remaining part of this section we will show that under certain conditions, a flag, balanced, CM simplicial complex may be \lq\lq modified\rq\rq~ such that it satisfies the hypothesis of Theorem \ref{dimcodimb}.
Let $\D$ be a CM, flag balanced  ($d-1$)-dimensional simplicial complex on $[n]$. As we have seen,  if $n = 2d$ and $\D$ has no cone points, we know that there exists a flag simplicial complex $\G$ such that $h(\D) = f(\G)$. Suppose now that $n> 2d$. Adding $n-2d$ cone points to $\D$ we still obtain a CM, flag and balanced simplicial complex, and the dimension of this new complex is equal to half the number of vertices. 

In order to simplify notation, suppose that  $\D$ is already a ($d-1$)-dimensional, CM, flag, balanced   simplicial complex on $[2d]$, with $r$ cone points $z_1,\ldots,z_r$. Let 
$[2d] = \cup_{i=1}^d V_i$
be the partition of the vertices corresponding to the coloring. Without loss of generality we may also assume  that $ V_{d+1 -j}=\{z_j\}$ for $j = 1,\ldots, r$. We will denote by $G = G_\D$ the graph of minimal nonfaces of $\D$. Suppose that $\D$ has  the  property that in $G$ for every $i \in 1,\ldots ,  d$ with $|V_i| >2$ we have
\begin{equation}\label{NP}
\exists ~y_{i,1}, y_{i,2} \in V_i \textup{~such that~} \forall~ x \in V_i \textup{~we have} ~N[y_{i,1}]\subseteq N[x] \textup{~or~}N[y_{i,2}]\subseteq N[x].
\end{equation}
Denote by $\overline{V_i} = V_i \setminus \{ y_{i,1},y_{i,2}\}$ ~
and  by $\overline{V} = \cup \overline{V_i}$ the union over all $i = 1,\ldots , (d-r)$ with $|V_i| > 2$. Notice that the cardinality of $\overline{V}$ satisfies $|\overline{V}| = r$, where $r$ is  the number of cone points. For any $x\in \overline{V}$ denote by $y_{x}$ the element of property \eqref{NP}. If for both $y_{i,1}$ and $y_{i,2}$ the inclusion of the closed neighborhoods is satisfied, then randomly choose one of them as $y_{x}$.
We will denote by $\Gens(I_\D)$ the set of minimal generators of the Stanley-Reisner ideal of $\D$. If no confusion may arise, we will denote the variables with the same letters as the vertices of $\D$. With the above notation we have:
\begin{lemma}
The flag simplicial complex $\widetilde{\D}$ corresponding to the square-free monomial ideal generated by
\begin{equation}\label{newdelta}
\left( \Gens(I_\D) \setminus \left( ~\bigcup_{x\in \overline{V}} \{xy_{x}\}\right)~\right) \cup \left(~\bigcup_{x\in \overline{V}}\{xz_j\}~\right) 
\end{equation}
is  balanced, Cohen-Macaulay and has the same $f$-vector as $\D$.
\end{lemma}
\begin{proof}
It is easy to see that it will be enough to  prove the lemma for $r=1$. We call a \lq\lq step\rq\rq~ the deletion of $xy_x$ from $\Gens(I_\D)$ together with the adding of $xz_j$ to $\Gens(I_\D)$ for some $x \in \overline{V}$. Notice that after \lq\lq taking a step\rq\rq~ property \eqref{NP} still holds in the new complex. To prove the lemma we have to show that at each step the $f$-vector does not change and that properties 1. and 2. below hold. It is clear that each step reduces $r$, the number of cone points,  by one. We will not need to prove Cohen-Macaulayness at each step, as it follows from properties 1. and 2. when there are no more cone points.

Suppose $r=1$ and that $i$ is the color for which  $|V_i|>2$, let  $x,y \in V_i$  be two vertices with $N[y] \subseteq N[x]$ and let $z$ be a cone point.

We will first prove that $f(\widetilde{\D})= f(\D)$. As $z$ is a cone point for $\D$, it will also be a cone point for the simplicial complex $\lk_\D x$. We will denote by $L_{xz} = \lk_\D\{x,z\}$.
By  definition $V(L_{xz})\cap N[x] = \emptyset$ , so property \eqref{NP} implies that $V(L_{xz}) \cap N[y] = \emptyset$ as well. 
This ensures that deleting the generator $xy$ we  obtain the new faces  $\widetilde{\D} \setminus \D = \{ F \cup \{x,y\} ~:~ F \in L_{xz}\}$. On the other hand, adding $xz$ as a generator we delete exactly the faces $ \{F \cup \{x,z\} ~:~ F\in L_{xz}\}$. This means we have for every $i \in \{-1,\ldots,d-1\}$:
\[ f_i(\widetilde{\D}) = f_i(\D) -f_{i-2}(L_{xz})+ f_{i-2}(L_{xz}),\]
where $f_j = 0 $ for $j<-1$.

Notice that $ \widetilde{\D}$ is still balanced. The only vertex that  changes in color is $x$, which will be colored with the same color as $z$. We will write $\cup_{i=1}^d \widetilde{V}_i$ for the partition of the vertices induced by the coloring.   In order to prove that $\widetilde{\D}$ remains CM we will prove that
\begin{compactenum}
\item $\widetilde{\D}$ is pure.
\item $\widetilde{\D}_S$ is a connected,  1-dimensional  complex for any subset of vertices $S = \widetilde{V}_i \cup \widetilde{V}_j$ with $1 \le i<j\le d$.
\end{compactenum}
Notice that (also for $r>1$) $\widetilde{\D}$ is a $(d-1)$-dimensional simplicial complex on $[2d]$, without cone points. It is easy to check that conditions 1. and 2. above imply the first point of Theorem \ref{structure} and thus  imply Cohen-Macaulayness. 

To prove 1.  we only have to check that the facets of the form $\{x,y\}\cup F$ with $F \in L_{xz}$ are of dimension $d-1$. But the maximal faces under inclusion in $L_{xz}$ are all of cardinality $d-2$ by the purity of $\D$, so $\widetilde{\D}$ is also pure.

To prove 2. we have to check three  cases. Fix $S = \widetilde{V}_i \cup \widetilde{V}_j$ with $1\le i <  j \le d$. \\
\mbox{\emph{Case 1.}~$S\cap \{x,y,z\} = \emptyset.$} In this case $\widetilde{\D}_S = \D_S$, so by  \cite[Theorem 4.5]{St} it is CM, thus connected.\\
\mbox{\emph{Case 2.}~$S\cap \{x,y,z\} = \{y\}$.} The inclusion $N[y] \subseteq N[x]$ is equivalent to $$\{v,x\} \in \D \Rightarrow \{v,y\} \in \D.$$ Let $v,w$ be two vertices in $\widetilde{\D}_S$. Again by \cite[Theorem 4.5]{St} in $\D_{S\cup\{x\}}$ there exists a path connecting them: $v=v_1,v_2,\ldots,v_t =w$. Suppose $v_j = x$ for some $j$. By the above observation $\{v_{j-1},y\}$ and $\{y,v_{j+1}\}$ are edges in $\widetilde{\D}_S$, so we can modify the path to $v_1, \ldots, v_{j-1},y,v_{j+1},\ldots,v_t$. Hence $\widetilde{\D}_S$ is also connected. \\
\mbox{\emph{Case 3.}~$S\cap \{x,y,z\} \supseteq \{x,z\}$.} Suppose $z \in \widetilde{V}_j$. As $z$ is a cone point in $\D$, it is connected to all vertices of $\widetilde{V}_i$. If $y\in \widetilde{V}_i$ then it is enough to notice that $\{x,y\} \in \widetilde{\D}_S$. Otherwise, as $\widetilde{\D}$ is pure and balanced, there exists at least one vertex $v\in \widetilde{V}_i$ such that $\{x,v\} $ is an edge.
\end{proof}

Using the above lemma together with Theorem \ref{dimcodimb} we obtain the following corollary.
\begin{corollary}
If $\D$ is a Cohen-Macaulay, flag, balanced simplicial complex satisfying property \eqref{NP}, there exists a flag simplicial complex $\G$ such that $h(\D) = f(\G)$.
\end{corollary}

In the next example we will see how Lemma \ref{newdelta} works.
\begin{example}
 Let $\D'$ be the flag, balanced simplicial complex corresponding to the graph on $\{1,\ldots,8\}$ represented on the left hand side.   Consider $\D$ to be the independence complex of the whole graph $G$ on $\{1,\ldots,10\}$. Notice that $\D$ is obtained from $\D'$ by adding the cone points $9$ and $10$. It is not difficult to check that $\D$ is CM, (actually  vertex decomposable). \\
\includegraphics[scale=0.7]{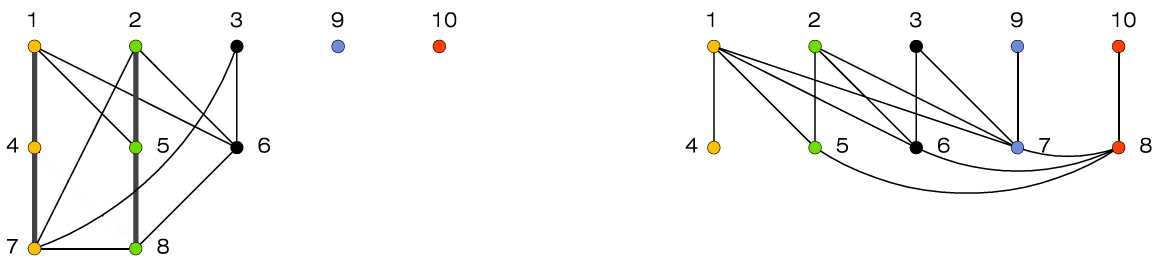}\\
Now we construct the simplicial complex $\widetilde{\D}$ as the independence complex of the graph $\widetilde{G}$ depicted on the right hand side. If we set $V_1 = \{1,4,7\}$ and $V_2=\{2,5,8\}$, using the notation of Lemma \ref{newdelta} we have $\overline{V} = \overline{V_1}\cup \overline{V_2} = \{7\} \cup \{8\}$. As $  V_1=N[4] \subseteq N[7] = V_1 \cup \{2,3,8\}$ and $V_2 \cup \{6,7\}=N[2]  \subseteq N[8] = V_2 \cup \{6,7\}$ we may choose $y_7 = 4$ and $y_8 =2$. Deleting the edges $\{4,7\}$ and $\{2,8\}$ and adding the edges $\{7,9\}$ and $\{8,10\}$ we find ourselves in the hypothesis of Lemma \ref{newdelta}, so $\widetilde{\D}$ is flag, balanced, CM and $f(\widetilde{\D}) = f(\D)$.
\end{example}

Unfortunately, property \eqref{NP} is not satisfied in general. The following simplicial complex turned up in several contexts as a counter-example to the strategy we were trying to use in order to prove Conjecture \ref{conjalexvabba}.
\begin{example}
Let $\D$ be the 2-dimensional simplicial complex on $[8]$ represented below on the left hand side. The picture on the right hand side represents the   graph  $G=G_\D$ of minimal nonfaces. \\

$\quad\quad\quad$\includegraphics[scale=0.8]{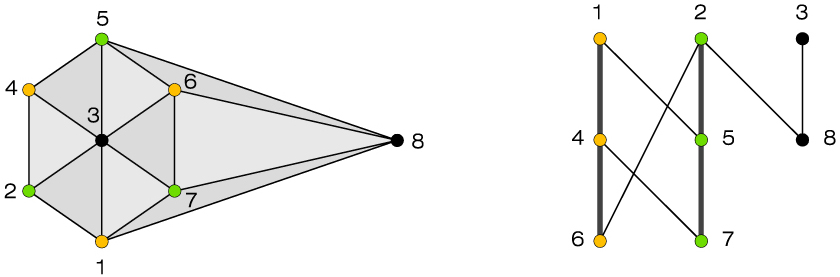}\\

Notice that $\D$ is balanced and it is also easy to check that it is vertex decomposable. One vertex decomposition is obtained by removing in order the vertices  $8,~ 7,~6,~5,~4$. 
Let $V_1 = \{1,4,6\},~ V_2=\{2,5,7\}$ and $V_3 = \{3,8\}$  the disjoint sets of vertices of the same color. Notice that these sets are uniquely determined, i.e. there is a unique 3-coloring modulo a permutation of the colors.   From $G$ we can easily read that  $N[1] = V_1 \cup \{5\}$, $N[4] = V_1 \cup \{7\}$  and $N[6] = V_1 \cup \{2\}$, so $\D$ does not satisfy property \eqref{NP}. It is also easy to check that $\D$ does not satisfy the conditions of Lemma \ref{coneface}.   However, $h(\D)= (1,5,3)$ is clearly the $f$-vector of a flag simplicial complex.

We would also like to notice that $\lk_\D 8$ is vertex decomposable, but its vertex decomposition cannot be induced by the vertex decomposition of $\D$, because $7$ is not a shedding vertex  for $\lk_\D 8$.
Notice that  for $\D\setminus 8$ both the lexicographic and the reversed lexicographic order on $\mathcal{F}(\D\setminus 8)$ are shelling orders. However, this is no longer true for $ (\D\setminus 8)_{\{1,7,6,5\}}$. 

The above observations also underline the fact  that even if vertex decomposability strongly encourages proofs by induction, in the case of Conjecture \ref{conjalexvabba}  this strategy works only in the presence of extra assumptions or leads to weaker conclusions.
\end{example}

The flag, balanced, pure simplicial complexes  with having property \eqref{cfp} are exactly the independence complexes of the clique-whiskered graphs introduced by Cook II and Nagel  in \cite{CN}. Both Lemma \ref{coneface} and Proposition \ref{cfCM} have a correspondent in the above mentioned paper.\\

%
%
%

\end{document}